\documentclass[a4paper,12pt,reqno]{amsart}
\usepackage{amsfonts}
\usepackage{amsmath}
\usepackage{amssymb}
\usepackage[a4paper]{geometry}
\usepackage{mathrsfs}
\usepackage{csquotes}
\usepackage{enumitem}

\usepackage{xcolor}
\usepackage[colorlinks = true,
            linkcolor = red,
            urlcolor  = gray,
            citecolor = blue,
            anchorcolor = blue]{hyperref}
\renewcommand\eqref[1]{(\ref{#1})} %Need with hyperref
\numberwithin{equation}{section}
\theoremstyle{plain}
\newtheorem{theorem}{Theorem}[section]
\newtheorem{proposition}[theorem]{Proposition}
\newtheorem{corollary}[theorem]{Corollary}
\newtheorem{lemma}[theorem]{Lemma}
\theoremstyle{definition}
\newtheorem{definition}[theorem]{Definition}
\newtheorem{remark}[theorem]{Remark}

\begin{document}

   \title[Variational eigenvalues of the subelliptic $p$-Laplacian]
   {Variational eigenvalues of the subelliptic \\$p$-Laplacian}

\author[M. Karazym]{Mukhtar Karazym}
\address{
  Mukhtar Karazym:
  \endgraf
  Department of Mathematics
  \endgraf
Nazarbayev University, Kazakhstan
  \endgraf
  {\it E-mail address} {\rm mukhtar.karazym@nu.edu.kz}
  }


\begin{abstract} 
We use the Lusternik-Schnirelman theory to prove the existence of a nondecreasing sequence of variational eigenvalues for the subelliptic $p$-Laplacian subject to the Dirichlet boundary condition. 
\end{abstract}
\maketitle

\section{Introduction}

Let $\Omega$ be a bounded connected open set in $\mathbb{R}^d$ with $d \geq 2$, $X_1, \ldots, X_m$ be a family of smooth H\"ormander vector fields with $m\leq d$. Unless otherwise stated, also let $1 < p < \infty$.   This paper is devoted to the study of variational eigenvalues of the quasilinear subelliptic equation in Carnot-Carath\'eodory spaces $(\Omega, d_X)$
\begin{equation}\label{Dirichlet quasilinear subelliptic}
\begin{aligned}
\sum_{i=1}^{m}X_{i}^{*}\left(\left|X u\right|^{p-2} X_{i}u\right)&=\lambda|u|^{p-2} u  &&\text { in } \Omega, \\
u&=0
&&\text { on } \partial \Omega,
\end{aligned}
\end{equation}
where $X_{i}^{*}$ is the formal adjoint of $X_{i}$, the control distance $d_X$ is induced by $X_{1},\ldots, X_{m}$. Particular cases of our results:

1. When $X_i=\frac{\partial}{\partial x_i}$ for all $1\leq i \leq m$ with $m=d$,  the problem \eqref{Dirichlet quasilinear subelliptic} becomes 
\begin{equation}\label{Dirichlet p-Laplacian}
\begin{aligned}
-\operatorname{div}\left(|\nabla u|^{p-2} \nabla u\right)&=\lambda |u|^{p-2} u\quad &&\text { in  } \Omega, \\
u &=0 &&\text { on } \partial \Omega.
\end{aligned}
\end{equation}
Variational eigenvalues of \eqref{Dirichlet p-Laplacian} were studied by many authors, see e.g.  \cite{DR1999, AA1987, Le2006, Perera2003}.

2. If $X_{1},\ldots,X_{m}$ are left invariant vector fields on $\mathbb{R}^d$ that generate the Lie algebra of some homogeneous Carnot group $\mathbb{G}=\left(\mathbb{R}^d, \circ, \delta_\lambda\right)$ with the group operation $\circ$ and dilation $\delta_\lambda$, the problem \eqref{Dirichlet quasilinear subelliptic}  becomes
\begin{equation}\label{Dirichlet p-sub-Laplacian}
\begin{aligned}
-\operatorname{div}_{\mathbb{G}}\left(\left|\nabla_{\mathbb{G}} u\right|^{p-2} \nabla_{\mathbb{G}} u\right)&=\lambda|u|^{p-2} u  &&\text { in } \Omega\subset \mathbb{G}, \\
u&=0 &&\text { on } \partial \Omega,
\end{aligned}
\end{equation}
where $\nabla_{\mathbb{G}}$ is the horizontal gradient and  $\operatorname{div}_{\mathbb{G}} v:=\nabla_{\mathbb{G}} \cdot v$ is the  horizontal divergence on $\mathbb{G}$. In this case, the control distance $d_X$ is equivalent to the homogeneous norm of the homogeneous Carnot group, see \cite[Proposition 5.1.4]{Bonfiglioli2007}.

3. Let $\beta\in \mathbb{Z}^{+}$. The following vector fields
\begin{equation}\label{Grushin x1}
X_1=\frac{\partial}{\partial x_1}, \enspace \ldots,\enspace X_{d-1}=\frac{\partial}{\partial x_{d-1}}\enspace \text{ and } \enspace X_{d}=x_{1}^\beta \frac{\partial}{\partial x_{d}}
\end{equation}
are smooth. Moreover, they satisfy the H\"ormander finite rank condition, however,  there does not exist any group law on $\mathbb{R}^{d}$ such that $X_{1},\ldots, X_{d}$ are left invariant. For the simplest case, we have
$$X_1=\frac{\partial}{\partial x_1}\enspace \text{ and }\enspace X_2=x_1\frac{\partial}{\partial x_2}\quad \text{on } \mathbb{R}^2, $$
see e.g. \cite[Example 1.2.14]{Bonfiglioli2007}. Therefore, we also focus on Carnot-Carathéodory spaces that are not homogeneous Carnot groups, commonly referred to as Grushin spaces.

\section{Preliminaries}
\subsection{H\"ormander vector fields}
Let $U$ be a bounded domain in $\mathbb{R}^{d}$  with $d\geq 2$ and let 
\begin{equation*}
X_{i}=\sum_{k=1}^{d} b_{ik}(x) \partial_{x_{k}}, \quad i=1,\ldots,m \enspace\text{ with }\enspace m\leq d,
\end{equation*}
be a collection of smooth vector fields in $U$. We set
$$X_{i}I(x):=\left(b_{i1}(x),\ldots,b_{id}(x)\right)^{\top}, \enspace x\in U,$$ 
for $i=1,\ldots,m$. Given a multi-index $J=\left(j_1, \ldots, j_s\right) \in\{1, \ldots, m\}^s$ with $s\in \mathbb{N}$, we define a commutator  of   $X_1, \ldots, X_m$ of length  $s=|J|$ as follows
\begin{equation*}
X_{J}:=\left[X_{j_1}, \ldots\left[X_{j_{s-1}}, X_{j_s}\right] \ldots\right].
\end{equation*}
If there is a smallest $s\in \mathbb{N}$ such that 
$\left\{X_{J}I(x)\right\}_{|J| \leq s}$ span $\mathbb{R}^{d}$ at each point $x\in U$, then we say that  the vector fields  $X_{1}, \ldots, X_{m}$ satisfy the H\"ormander rank condition at step $s$. 

Let $\Omega \Subset U$ be an open connected subset. Also, let $\delta > 0$.
\begin{definition}
We define $C_1(\delta)$ as the class of absolutely continuous curves $\varphi: [0,1] \to \Omega$ that satisfy the equation
$$
\varphi^{\prime}(t)=\sum_{i=1}^{m} a_{i}(t)X_{i}I\left({\varphi(t)}\right)\quad  \text {a.e. in } [0,1],
$$
where  $a_i$ are measurable functions such that
$$
\left|a_{i}\right| \leq \delta \quad \text {a.e. in } [0,1].
$$
Now, let $x, y \in \Omega$. The control distance is defined as
\begin{equation*}
d_{X}(x, y):=\inf \left\{\delta>0:\enspace \exists \varphi \in C_{1}(\delta) \text { with } \varphi(0)=x \text{ and } \varphi(1)=y\right\}.
\end{equation*}
\end{definition}
By the connectivity theorem \cite[Theorem 1.45]{bramanti2023hormander} (see also \cite[Theorem 19.1.3]{Bonfiglioli2007} in the case of stratified vector fields), the control distance \( d_X \) is finite on \( \Omega \). This result, known as the Chow--Rashevsky theorem, was first established in \cite{chow1940systeme, rashevsky1938any}. Hence, the pair \( (\Omega, d_X) \) forms a metric space, often referred to as a Carnot--Carathéodory space.

Next, we  define the Sobolev space $\mathcal{W}_{X}^{1, p}(\Omega)$ induced by  $X_{1}, \ldots, X_{m}$.  We begin by introducing the concept of weak derivatives with respect to the Hörmander vector fields. A function $f\in L^{1}_{loc}(\Omega)$ is said to be differentiable in the weak sense with respect to  $X_{i}$, if there exists  $g\in L^{1}_{loc}(\Omega)$ such that
\begin{equation*}
\int_{\Omega} g(x)  \varphi(x) d x=\int_{\Omega} f(x) X_{i}^{*} \varphi(x) d x\quad \text{ for all } \varphi \in C_{0}^{\infty}(\Omega),
\end{equation*}
where $X_{i}^{*}$ denotes the formal adjoint of  $X_{i}$ defined by
$$
X_{i}^{*}\varphi(x)=-\sum_{k=1}^d \partial_{x_{k}}\left(b_{ik} \varphi(x)\right).
$$
For  $f\in L^{1}_{loc}(\Omega)$, we denote the weak derivative with respect to $X_i$ by $X_i f=g$. The Sobolev space induced by the vector fields $X_{1}, \ldots, X_{m}$ is then defined as
\begin{equation*}
\mathcal{W}_{X}^{1, p}(\Omega):=\left\{f \in L^{p}(\Omega):\enspace X_{i} f \in L^{p}(\Omega)\quad \text{for } i=1, \ldots, m \right\}
\end{equation*}
with the norm
\begin{equation}\label{natural norm}
\|u\|_{\mathcal{W}_{X}^{1, p}(\Omega)}:=\left(\int_{\Omega}\left(|u|^p+|X u|^p\right) d x\right)^{\frac{1}{p}},
\end{equation}
where $X:=\left(X_{1}, \ldots, X_{m}\right)$ is referred to as the horizontal gradient, and its length is given by
$$|X f|=\left(\sum_{i=1}^{m}\left(X_{i} f\right)^{2}\right)^{\frac{1}{2}}.$$
Next, we define the trace zero Sobolev space  $\mathcal{W}_{X,0}^{1, p}(\Omega)$, which is the closure of $C_{0}^{\infty}(\Omega)$ in $\mathcal{W}_{X}^{1, p}(\Omega)$. 
Let us review some known properties of $\mathcal{W}_{X,0}^{1, p}(\Omega)$. Since $\mathcal{W}_{X,0}^{1, p}(\Omega)$ is a  reflexive Banach space (see e.g. \cite[Theorem 1]{xu1990subelliptic}),  every bounded sequence in $\mathcal{W}_{X,0}^{1, p}(\Omega)$ admits a weakly convergent subsequence. Therefore, it is important to characterize weak convergence in $\mathcal{W}_{X,0}^{1, p}(\Omega)$.
\begin{proposition}\label{Weak convergence characterization}
A sequence $\{v_n\}$ converges weakly to some function $v$ in $\mathcal{W}_{X,0}^{1, p}(\Omega)$ if and only if there exist  $g_{i}\in L^{p}(\Omega)$ such that 
$$v_n\rightharpoonup v \enspace\text{ weakly in } L^{p}(\Omega) \quad \text{ and } \quad X_{i}v_n\rightharpoonup g_{i}\enspace \text{ weakly in } L^{p}(\Omega)$$ 
for $i=1,\ldots,m$. Here $g_{i}=X_{i}v$.
\end{proposition}
\begin{proof}
The proof follows standard methods; we refer to \cite[Proposition 1.8]{LeDret2018} and \cite[Corollary 11.70]{Leoni2017Sobolev} for further details. Additionally, see \cite[Proposition 2.2]{capogna2024asymptotic}.
\end{proof}
Another important tool in our framework is the following Poincar\'e-Friedrichs inequality for H\"ormander vector fields.
\begin{theorem} \cite{capogna1993embedding} There exists a  constant $C>0$ such that 
\begin{equation}\label{Poincare inequality}
\int_{\Omega}|u|^{p} d x\leq C \int_{\Omega}|X u|^{p} d x \qquad \text{for all }u\in \mathcal{W}_{X,0}^{1, p}(\Omega).
\end{equation}
See also \cite[Theorem 2.5]{capogna2024asymptotic}. 
\end{theorem}
By virtue of the Poincar\'e-Friedrichs inequality \eqref{Poincare inequality}, the space $\mathcal{W}_{X,0}^{1, p}(\Omega)$ can be endowed with the equivalent norm
\begin{equation}\label{Poincare norm}
\|u\|:=\left(\int_{\Omega}\left|X u\right|^{p} d x\right)^{\frac{1}{p}}.
\end{equation}
The notation $\|\cdot\|$ without a subscript is only used for \eqref{Poincare norm} in the sequel. Moreover, $\mathcal{W}_{X,0}^{1, p}(\Omega)$ equipped with \eqref{Poincare norm}  is uniformly convex. For the sake of completeness, we provide the proof.
\begin{theorem}
The space $\mathcal{W}_{X,0}^{1, p}(\Omega)$ equipped with \eqref{Poincare norm}  is uniformly convex.
\end{theorem}
\begin{proof}
Let us split the proof into two cases $1<p<2$ and $p\geq 2$. First, let $p\geq 2$ and let $u,v\in \mathcal{W}_{X,0}^{1, p}(\Omega)$ satisfy 
$$\|u\|=\|v\|=1 \enspace\text{and} \enspace\|u-v\|\geq \varepsilon \enspace \text{ for }\varepsilon\in (0,2].$$
Using the vector inequality \eqref{Adams p>2} for  $\omega_{1}=Xu$ and $\omega_{2}=Xv$ and integrating over $\Omega$, we obtain 
$$
\begin{aligned}
\left\|\frac{u+v}{2}\right\|^p+\left\|\frac{u-v}{2}\right\|^p  &=\int_{\Omega}\left(\left|\frac{X u+X v}{2}\right|^p+\left|\frac{X u-X v}{2}\right|^p\right)dx \\
 &\leq \frac{1}{2} \int_{\Omega}\left(|X u|^p+|X v|^p\right)dx=\frac{1}{2}\left(\|u\|^p+\|v\|^p\right)=1.
\end{aligned}
$$
Then 
$$
\left\|\frac{u+v}{2}\right\|^p \leq 1-\left(\frac{\varepsilon}{2}\right)^p .
$$
Therefore, there exists $\delta=\delta(\varepsilon)>0$ such that
$$
\left\|\frac{u+v}{2}\right\|^p \leq 1-\delta.
$$
Now, let $1<p<2$.  Let $u,v\in \mathcal{W}_{X,0}^{1, p}(\Omega)$ satisfy 
$$\|u\|=\|v\|=1 \enspace\text{and} \enspace\|u-v\|\geq \varepsilon \enspace \text{ for }\varepsilon\in (0,2].$$ 
It is clear that $|Xu|^{p^{\prime
}},|Xv|^{p^{\prime
}}\in L^{p-1}(\Omega)$, 
$$\||Xu|^{p^{\prime
}}\|_{p-1}=\|u\|^{p^{\prime}}\enspace \text{and}\enspace \||Xv|^{p^{\prime
}}\|_{p-1}=\|v\|^{p^{\prime}},$$ 
where $p^{\prime}=\frac{p}{p-1}$. Then
$$
\begin{aligned}
\left\|\frac{u+v}{2}\right\|^{p^{\prime}}+\left\|\frac{u-v}{2}\right\|^{p^{\prime}} & =\left\|\left| \frac{Xu+Xv}{2}\right|^{p^{\prime}}\right\|_{p-1}+\left\|\left| \frac{Xu-Xv}{2}\right|^{p^{\prime}}\right\|_{p-1} \\
& \leq\left\|\left|\frac{Xu+Xv}{2}\right|^{p^{\prime}}+\left| \frac{Xu-Xv}{2}\right|^{p^{\prime}}\right\|_{p-1} \\
& =\left[\int_{\Omega}\left(\left|\frac{X u+X v}{2}\right|^{p^{\prime}}+\left|\frac{X u-X v}{2}\right|^{p^{\prime}}\right)^{p-1}\right]^{\frac{1}{p-1}} \\
& \leq\left[\frac{1}{2} \int\left(\left|X u\right|^p+\left|X v\right|^p\right)\right]^{\frac{1}{p-1}} \\
& =\left[\frac{1}{2}\left\|u\right\|^p+\frac{1}{2}\left\|v\right\|^p\right]^{\frac{1}{p-1}}=1,
\end{aligned}
$$
where we have used  the reverse Minkowski inequality
$$
\|\left|X u\right|^{p^{\prime}}+\left|X v\right|^{p^{\prime}}\|_{p-1} \geq\|\left|X u\right|^{p^{\prime}}\|_{p-1}+\|\left|X v\right|^{p^{\prime}}\|_{p-1},
$$
and the vector inequality      \eqref{Adams 1<p<2}
for $\omega_{1}=Xu$ and $\omega_{2}=Xv$. So, we have
$$
\left\|\frac{u+v}{2}\right\|^{p^{\prime}} \leq 1-\left(\frac{\varepsilon}{2}\right)^{p^{\prime}}.
$$
Therefore, there exists $\delta=\delta(\varepsilon)>0$ such that
$$
\left\|\frac{u+v}{2}\right\|^p \leq 1-\delta.
$$
\end{proof}
We conclude this subsection with the following compact Sobolev embedding theorem.
\begin{theorem}\label{compact Sobolev embedding}\cite[Corollary 3.3]{Danielli1991}
Let $1\leq p<\infty$. Then  $\mathcal{W}_{X,0}^{1, p}(\Omega)\hookrightarrow L^{p}(\Omega)$ is compact.
\end{theorem}

\subsection{Lusternik-Schnirelmann theory}
Let $\mathbb{B}$ be an infinite dimensional real reflexive Banach space. Its dual space is denoted by $\mathbb{B}^{*}$. For all $\omega \in \mathbb{B}^{*}$ and $u\in \mathbb{B}$, we set a dual pair $\left\langle \omega, u\right\rangle=\omega(u)$. Strong convergence in  $\mathbb{B}$ and $\mathbb{B}^{*}$ is denoted by $\rightarrow$. Weak convergence in  $\mathbb{B}$ and $\mathbb{B}^{*}$ is denoted by $\rightharpoonup$.

\begin{definition}
We say that an operator $T:\mathbb{B}\to \mathbb{B}^{*}$ satisfies condition $(S)_{0}$, if $\{u_n\}$ in $\mathbb{B}$ satisfies
\begin{equation*}
\begin{aligned}
u_n &\rightharpoonup u \quad &&\text{in }\mathbb{B},\\
Tu_n &\rightharpoonup v \enspace &&\text{in }\mathbb{B}^{*}\\
\left\langle Tu_n, u_n\right\rangle &\rightarrow\langle v, u\rangle\enspace && \text{in }\mathbb{R},\\
\end{aligned},
\end{equation*}
then $u_n \rightarrow u$ in $\mathbb{B}$.
\end{definition}

\begin{definition}
Let $M \subset \mathbb{B}\setminus\{0\}$ be a
symmetric compact subset. We say that the set $M$ has genus $n$, denoted by $\gamma(M)=n$, if there exists an odd continuous operator $h: M \rightarrow \mathbb{R}^{n}\setminus\{0\}$, where $n\in\mathbb{N}$ is the smallest  number holding this property. If there is no such natural number $n$, then $\gamma(A)=\infty$. We call $\gamma$ the Krasnoselskii genus. 
\end{definition}

Let $F, G:\mathbb{B}\to \mathbb{R}$. It is known that the following eigenvalue problem
\begin{equation}\label{eigenvalue problem F and G}
F^{\prime}(u)=\mu G^{\prime}(u)\quad\text{on}\enspace\mathcal{G}
\end{equation}
can be treated by the Lusternik-Schnirelman theory (see e.g. \cite[Section 44.5]{Zeidler1985}).
Here $\mu \in \mathbb{R}$ and  $\mathcal{G}$ is the level set defined by
\begin{equation}\label{level set}
\mathcal{G}:=\{u \in \mathbb{B}: G(u)=1\}.    
\end{equation}

\begin{proposition}(\cite[Proposition  43.21]{Zeidler1985}).
A function $u\in  \mathbb{B}\setminus\{0\}$ is an eigenfunction of \eqref{eigenvalue problem F and G} if and only if it is a critical point of the functional $F$ with respect to the level set $\mathcal{G}$.
\end{proposition}

Assume that
\begin{enumerate}[font={\bfseries},label=({H\arabic*})]
\item $F,G\in C^{1}\left( \mathbb{B},\mathbb{R}\right)$  are even functionals such that $F(0)=G(0)=0$;

\item $F^{\prime}$ is strongly continuous. Moreover, 
\begin{equation}\label{H2}
\left\langle F^{\prime}(u), u\right\rangle=0, \enspace u\in \overline{\operatorname{conv}}\mathcal{G} \quad  \Longrightarrow \quad   F(u)=0,
\end{equation}
where $\overline{\operatorname{conv}}\mathcal{G}$ denotes the closed convex hull of $\mathcal{G}$;

\item The functional $G^{\prime}$ is continuous, bounded and satisfies condition $(S)_{0}$;

\item The level set $\mathcal{G}$ is bounded and
\begin{equation*}
u \neq 0 \quad\Longrightarrow \quad\begin{aligned}
&\left\langle G^{\prime}(u), u\right\rangle>0;\\
&\lim _{t \rightarrow+\infty} G(t u)=+\infty;\\
&\inf _{u \in \mathcal{G}}\left\langle G^{\prime}(u), u\right\rangle>0.
\end{aligned}
\end{equation*}
\end{enumerate}

We define
\begin{equation*}
\mathcal{G}_n:=\left\{\mathcal{A} \subset \mathcal{G}: \mathcal{A} \right. \text{ is compact and symmetric, }\left.\gamma(\mathcal{A}) \geq n\right\}. 
\end{equation*}
Also let
$$
\beta_n:= \begin{cases}\sup _{\mathcal{A}  \in \mathcal{G}_n} \inf _{u \in \mathcal{A} } F(u) & \text { if}\enspace\mathcal{G}_n \neq \varnothing,   \\ 0 &\text { if}\enspace \mathcal{G}_n=\varnothing. \end{cases}
$$
and
$$
\chi:= \begin{cases}\sup \left\{n \in \mathbb{N}: \beta_n>0\right\} & \text { if}\enspace \beta_1>0, \\ 0 & \text { if}\enspace \beta_1=0.\end{cases}
$$

The following theorem is the Lusternik-Schnirelman principle, see e.g. \cite[Theorem 44.A]{Zeidler1985}.
\begin{theorem}\label{H1-H4 theorem}
Under assumptions \textbf{(H1)-(H4)}, we have
\begin{enumerate}[label=(\roman*)]
    \item If $\beta_n$ is positive, then the eigenvalue problem \eqref{eigenvalue problem F and G} admits  eigenpairs $(\mu_n,\,u_n)$ with $\mu_{n}\neq 0$ and $F\left(u_n\right)=\beta_n$;
        \item If $\chi=\infty$, then the eigenvalue problem \eqref{eigenvalue problem F and G} admits infinitely many eigenfunctions associated with nonzero eigenvalues;
    \item The sequence $\{\beta_{n}\}$ satisfies
    \begin{equation*}
    \infty>\beta_1 \geq \beta_2 \geq \ldots \geq 0\quad \text{and}\quad \beta_n \rightarrow 0 \enspace\text{as}\enspace n \rightarrow \infty;
    \end{equation*}
   
    \item If $\chi=\infty$ and 
    \begin{equation}\label{Zeidler 44.A (3)}
    F(u)=0, \enspace u\in \overline{\operatorname{conv}}\mathcal{G} \quad  \Longrightarrow \quad \left\langle F^{\prime}(u), u\right\rangle=0,
    \end{equation}
    then there exists a sequence of distinct eigenvalues $\{\mu_{n}\}$ such that $\mu_{n}\to 0$ as $n\to \infty$;
      \item Suppose
    $$F(u)=0 \enspace\text{and} \enspace u \in \overline{\operatorname{conv}}\mathcal{G}\quad \Longrightarrow\quad u=0.$$
    Then $\chi=\infty$ and there exists a sequence of eigenpairs $\left\{\left(u_n, \mu_n\right)\right\}$ of \eqref{eigenvalue problem F and G} such that
        \begin{equation*}
      u_n \rightharpoonup 0\quad \text{and}\quad   \mu_n \rightarrow 0\quad \text{as} \quad n \rightarrow \infty,
    \end{equation*}
    where $\mu_n \neq 0$ for all $n\in \mathbb{N}$.
\end{enumerate}
\end{theorem}

\begin{remark}
For assumptions \textbf{(H2)-(H3)}, we refer to \cite[Remark 44.23]{Zeidler1985}.
\end{remark}

\section{Variational eigenvalues of the Lusternik-Schnirelmann type }
We consider the nonlinear eigenvalue problem
\begin{equation}\label{eigenvalue problem}
\begin{aligned}
\sum_{i=1}^{m}X_{i}^{*}\left(\left|X u\right|^{p-2} X_{i}u\right)&=\lambda|u|^{p-2} u  &&\text { in } \Omega, \\
u&=0
&&\text { on } \partial \Omega,
\end{aligned}
\end{equation}
in the weak sense.

\begin{definition}\label{def:eigenfunction}
We say that $\lambda$  is an eigenvalue of \eqref{eigenvalue problem}, if there is $u \in \mathcal{W}_{X,0}^{1, p}(\Omega)\setminus \{0\}$  such that
\begin{equation}\label{eigenfunction}
\int_{\Omega}|X u|^{p-2} Xu \cdot X \varphi d x=\lambda \int_{\Omega}|u|^{p-2} u \varphi d x
\end{equation}
holds for all $\varphi\in \mathcal{W}_{X,0}^{1, p}(\Omega)$.
\end{definition}
The first eigenvalue of \eqref{eigenvalue problem} is given by
\begin{equation*}
\lambda_{1}:=\inf_{v \in \mathcal{W}_{X,0}^{1, p}(\Omega)\setminus\{0\}}\frac{\int_{\Omega}|X v|^{p} d x}{\int_{\Omega}|v|^{p} d x}=\frac{\int_{\Omega}|X u_{1}|^{p} d x}{\int_{\Omega}|u_{1}|^{p} d x},
\end{equation*}
where $u_{1}$ is the corresponding (first) eigenfunction. Recently in \cite{KS2023}, it is proved that $\lambda_{1}$ is simple and isolated. Also, they showed that $u_{1}$ is positive in $\Omega$ and all eigenfunctions of \eqref{eigenvalue problem} are H\"older continuous with respect to the control distance.

In this section, following \cite[Section 44.5]{Zeidler1985}, we show the existence of variational eigenvalues. Let $F,G: \mathcal{W}_{X,0}^{1, p}(\Omega)\to \mathbb{R}$ defined by
\begin{equation*}
F(u):=\int_{\Omega} |u|^p d x
\quad \text{and}\quad
G(u):=\int_{\Omega} |X u|^p d x.
\end{equation*}
Then the level set $\mathcal{G}$ defined by \eqref{level set} is the unit sphere in $\mathcal{W}_{X,0}^{1, p}(\Omega)$.

The functionals $F$ and $G$ are  continuously differentiable. Their (G\^ateaux) derivatives are given by
\begin{equation*}
\langle F^{\prime} (u), v\rangle=p\int_{\Omega} |u|^{p-2} u v d x
\end{equation*}
and
\begin{equation*}
\langle G^{\prime} (u), v\rangle=p\int_{\Omega} |X u|^{p-2} X u \cdot X v d x,
\end{equation*}
where $v\in \mathcal{W}_{X,0}^{1, p}(\Omega)$. Since $\langle G^{\prime} (u), u\rangle=pG(u)$, then zero is the  only one critical value of $G$. Therefore, the level set $\mathcal{G}$ is a $C^1$-Finsler manifold by the implicit function theorem.

Then \eqref{eigenfunction} converts to the following form
\begin{equation*}
\langle F^{\prime}(u),\varphi\rangle=\mu \langle G^{\prime}(u),\varphi\rangle \quad  \text{for all } \varphi\in \mathcal{W}_{X,0}^{1, p}(\Omega),
\end{equation*}
or in short
\begin{equation}\label{operator eigenvalue problem}
F^{\prime}(u)=\mu G^{\prime}(u) \qquad  \text{on } \mathcal{G},
\end{equation}
where $\mu=\lambda^{-1}>0$.

In order to apply Theorem \ref{H1-H4 theorem}, we start verifying \textbf{(H2)} and \textbf{(H3)}. The assumptions \textbf{(H1)} and \textbf{(H4)} are trivial.

\begin{proposition}\label{H2 property}
The operator $F^{\prime}$ satisfies \textbf{(H2)}.
\end{proposition}
\begin{proof}
Since \eqref{H2} follows from $\langle F^{\prime} (u), u\rangle=pF(u)$, it remains to prove the strong continuity of $F^{\prime}$. Let  $u_n \rightharpoonup u$ in $\mathcal{W}_{X,0}^{1, p}(\Omega)$. Then  we have 
\begin{equation*}
\left|\left\langle F^{\prime} (u_n)-F^{\prime} (u), v\right\rangle\right|  \leq p\left\|\left|u_n\right|^{p-2} u_n-|u|^{p-2} u\right\|_{p'}\|v\|_{p}\quad \text{for all } v \in \mathcal{W}_{X,0}^{1, p}(\Omega)
\end{equation*}
by the Hölder inequality, where $p'=\frac{p}{p-1}$.
If we prove that
\begin{equation}\label{H2 property: convergence}
 \left|u_n\right|^{p-2} u_n \rightarrow|u|^{p-2} u  \quad \text{in}\enspace  L^{p'}(\Omega),
\end{equation}
then $$\|F^{\prime}(u_{n})-F^{\prime}(u)\|_{{\mathcal{W}_{X,0}^{1, p}(\Omega)}^{*}}=\sup_{\stackrel{\|v\|\leq 1}{v\in \mathcal{W}_{X,0}^{1, p}(\Omega)}}\left|\left\langle F^{\prime} (u_n)-F^{\prime} (u), v\right\rangle\right|\to 0 \enspace\text{ as }\enspace n\to \infty.$$
By Proposition \ref{Weak convergence characterization} and Theorem \ref{compact Sobolev embedding}, it follows that
\begin{equation*}
 X u_{n} \rightharpoonup X u \quad \text { in } L^p(\Omega)\quad \text { and } \quad u_{n}\rightarrow u \quad\text { in } L^p(\Omega).
\end{equation*} 
Now we consider $1<p<2$ and $p\geq 2$ cases separately. Let $1<p<2$. Then the vector inequality \eqref{GM1975 p<2} yields
\begin{equation}\label{F' case 1}
    \begin{aligned}
  \int_{\Omega}\left|\left|u_n\right|^{p-2} u_n-|u|^{p-2} u\right|^{p'}dx&\leq C\int_{\Omega}\left|u_n- u\right|^{p'(p-1)}dx\\
  & = C\left\|u_n- u\right\|_{p}^{p}.
    \end{aligned}
\end{equation}
Since $u_{n}\rightarrow u$ in $L^p(\Omega)$, we see that the left-hand side of \eqref{F' case 1} tends to zero.

Let $p\geq 2$. Then combining the vector inequality \eqref{GM1975 p>2}, H\"older's inequality with exponent $p-1$ and the triangle inequality, we obtain
\begin{equation}\label{F' case 2}
    \begin{aligned}
  \int_{\Omega}\left|\left|u_n\right|^{p-2} u_n-|u|^{p-2} u\right|^{p'}dx&\leq C\int_{\Omega}\left|u_n- u\right|^{p'}\left(\left|u_n\right|+|u|\right)^{p'(p-2)}dx\\
  & \leq C \left\|u_n- u\right\|_{p}^{p'}\left(\left\|u_n\right\|_{p}+\|u\|_{p}\right)^{p'(p-2)}.
    \end{aligned}
\end{equation}
Since $u_{n}\rightarrow u$ in $L^p(\Omega)$, we see that the left-hand side of \eqref{F' case 2} tends to zero. We conclude that $$F^{\prime} (u_n) \rightarrow F^{\prime} (u) \quad \text{ in } {\mathcal{W}_{X,0}^{1, p}(\Omega)}^{*}.$$

\end{proof}

The following proposition is necessary  to verify \textbf{(H3)}. 
\begin{proposition}\label{B proposition}
Let $u,v\in \mathcal{W}_{X,0}^{1, p}(\Omega)$. Then
\begin{equation*}
\langle G^{\prime} (u)-G^{\prime} (v), u-v\rangle \geq p \left(\|u\|^{p-1}-\|v\|^{p-1}\right)\left(\|u\|-\|v\|\right).
\end{equation*}
Moreover, $\langle  G^{\prime} (u)-G^{\prime} (v), u-v\rangle=0$ if and only if $u=v$ a.e. in $\Omega$.
\end{proposition}
\begin{proof}
Given $u,v\in \mathcal{W}_{X,0}^{1, p}(\Omega)$, we have
\begin{equation*}
\begin{aligned}
\langle  G^{\prime} (u)-G^{\prime} (v), u-v\rangle&=p\Big(\|u\|^p+\|v\|^p  \\
&-\int_{\Omega} |X u|^{p-2} X u \cdot X v d x-\int_{\Omega} |X v|^{p-2} X v \cdot X u d x \Big).
\end{aligned}
\end{equation*}
Using the  H\"older inequality for the last two terms, we obtain
\begin{equation*}
\langle G^{\prime} (u)-G^{\prime} (v), u-v\rangle  \geq p\left(\|u\|^{p-1}-\|v\|^{p-1}\right)(\|u\|-\|v\|).
\end{equation*}
Now, let $u$ and $v$ satisfy 
\begin{equation}\label{B equality}
\langle G^{\prime} (u)-G^{\prime} (v), u-v\rangle=0.
\end{equation}
We consider $1<p<2$ and $p\geq 2$ cases separately. Let $p\geq 2$.  Setting $\omega_1=Xu$ and $\omega_2=Xv$, we integrate both sides of \eqref{vector inequality p>2} over $\Omega$ to get
\begin{equation}\label{strict monotonicity p>2}
0=\langle G^{\prime} (u)-G^{\prime} (v), u-v\rangle \geq C\|u-v\|^p,
\end{equation}
which implies $u=v$ a.e. in $\Omega$.

Let $1<p<2$. Then by the H\"older inequality with exponent $2/p$ and the vector inequality \eqref{vector inequality p<2}, we obtain
\begin{equation}\label{strict monotonicity 1<p<2}
\begin{aligned}
\int_{\Omega}|Xu-X v|^{p}dx&=\int_{\Omega}\frac{|Xu-X v|^{p}}{\big(|Xu|+|Xv|\big)^{\frac{p(2-p)}{2}}}\big(|Xu|+|Xv|\big)^{\frac{p(2-p)}{2}}dx
\\
&\leq\left(\int_{\Omega}\frac{|Xu-X v|^{2}}{\big(|Xu|+|Xv|\big)^{2-p}}dx \right)^{\frac{p}{2}}
\left(\int_{\Omega}\big(|Xu|+|Xv|\big)^{p}dx \right)^{\frac{2-p}{2}}
\\
&\leq \left(\int_{\Omega}\frac{|Xu-X v|^{2}}{\big(|Xu|+|Xv|\big)^{2-p}}dx \right)^{\frac{p}{2}}\big(\|u\|+\|v\|\big)^{\frac{p(2-p)}{2}}\\
&\leq \left(\int_{\Omega}\left(|X u|^{p-2} X u-|X v|^{p-2} X v\right) \cdot(X u-X v) d x\right)^{\frac{p}{2}}\big(\|u\|+\|v\|\big)^{\frac{p(2-p)}{2}}\\
&= \left\langle G^{\prime}(u)-G^{\prime}(v), u-v\right\rangle^{\frac{p}{2}}\big(\|u\|+\|v\|\big)^{\frac{p(2-p)}{2}}=0.
\end{aligned}
\end{equation}
Hence, $u=v$ a.e. in $\Omega$.
\end{proof}

\begin{proposition}\label{proposition H3}
The operator $G^{\prime}$ satisfies \textbf{(H3)}.
\end{proposition}
\begin{proof}
Using the H\"older inequality, we obtain
\begin{equation*}
\langle G^{\prime}( u), v\rangle  \leq p\|u\|^{p-1}\|v\|.
\end{equation*}
It shows that $G^{\prime}$ is bounded.
    
Let ${u_{n}}\to u$ in $\mathcal{W}_{X,0}^{1, p}(\Omega)$. To prove the continuity of $G^{\prime}$, first we  apply the H\"older inequality
\begin{equation*}
\left|\left\langle G^{\prime} (u_n)-G^{\prime} (u), v\right\rangle\right|\leq\left(\int_{\Omega} \Big||X u_n|^{p-2} X u_n-|X u|^{p-2} X u\Big|^{p'} d x\right)^{\frac{1}{p'}}\|v\|,
\end{equation*}
where $p'=\frac{p}{p-1}$. If we show that
\begin{equation}\label{H3 property: convergence}
 \left|Xu_n\right|^{p-2} Xu_n \rightarrow|Xu|^{p-2} Xu  \quad \text{in}\quad  L^{p'}(\Omega),
\end{equation}
then $G^{\prime}$ is continuous. Setting $w_n:=\left|X u_n\right|^{p-2} X u_n$ and $w:=$ $|X u|^{p-2} X u$, one can prove \eqref{H3 property: convergence} by using Lemma \ref{Lemma 5.1}.

Now, let $\{u_n\}\subset \mathcal{W}_{X,0}^{1, p}(\Omega)$ satisfy
\begin{equation*}
\begin{aligned}
u_n &\rightharpoonup u &&\text{in}\enspace \mathcal{W}_{X,0}^{1, p}(\Omega),\\
G^{\prime}(u_{n}) &\rightharpoonup v  &&\text{in}\enspace {\mathcal{W}_{X,0}^{1, p}(\Omega)}^{*},\\
\left\langle G^{\prime}(u_{n}), u_n\right\rangle &\rightarrow\langle v, u\rangle \quad&& \text{in}\enspace\mathbb{R},\\
\end{aligned}
\end{equation*}
for some $v \in {\mathcal{W}_{X,0}^{1, p}(\Omega)}^{*}$ and $u \in \mathcal{W}_{X,0}^{1, p}(\Omega)$.
It is sufficient to prove $\|u_n\|\to \|u\|$, since  $\mathcal{W}_{X,0}^{1, p}(\Omega)$ possesses the Radon-Riesz property due to the uniform convexity, that is,
\begin{equation*}
\|u_n\|\to \|u\|\quad \text{and}\quad u_n \rightharpoonup u \enspace \text{ in }\mathcal{W}_{X,0}^{1, p}(\Omega)
\enspace \Longrightarrow \enspace u_n \rightarrow u \enspace \text{ in }\mathcal{W}_{X,0}^{1, p}(\Omega).
\end{equation*}
On the one hand, we have
\begin{equation*}
\begin{aligned}
&\lim _{n \rightarrow \infty}\left\langle G^{\prime}(u_n)- G^{\prime}(u), u_n-u\right\rangle\\
&=\lim _{n \rightarrow \infty}\left(\left\langle G^{\prime}\left(u_n\right), u_n\right\rangle-\left\langle G^{\prime}\left(u_n\right), u\right\rangle-\left\langle G^{\prime}(u), u_n-u\right\rangle\right)=0.
\end{aligned}
\end{equation*}
On the other hand, Proposition \ref{B proposition} ensures
\begin{equation*}
\left\langle G^{\prime}(u_n)- G^{\prime}(u), u_n-u\right\rangle \geq p\left(\left\|u_n\right\|^{p-1}-\|u\|^{p-1}\right)\left(\left\|u_n\right\|-\|u\|\right).
\end{equation*}
Both imply $\left\|u_n\right\| \rightarrow\|u\|$.
\end{proof}
Now we are ready to state our main theorem.
\begin{theorem}
There exists a nonincreasing sequence of  eigenvalues $\left\{\mu_{n}\right\}$ of \eqref{operator eigenvalue problem}, given by
\begin{equation}\label{mu n}
\mu_n:=\sup _{\mathcal{A} \in \mathcal{G}_n} \inf _{u \in \mathcal{A}} F(u).
\end{equation}
Moreover, $\mu_{n} \rightarrow 0^{+}$ as $n \rightarrow \infty$.
\end{theorem}
\begin{proof}
Since $\mathcal{G}$ is the unit sphere in $\mathcal{W}_{X,0}^{1, p}(\Omega)$, we have
 $\gamma(\mathcal{G})=\dim \mathcal{W}_{X,0}^{1, p}(\Omega)= \infty$ by \cite[Proposition 44.10]{Zeidler1985}.   Continuity of $G$ implies that $\mathcal{G}=G^{-1}(\{1\})$ is closed. Let $V_n$ be a $n$-dimensional subspace of $\mathcal{W}_{X,0}^{1, p}(\Omega)$. Then it is closed. Therefore, $\mathcal{G}\cap V_n$ is   compact. Moreover, $\mathcal{G}\cap V_n$ is the unit sphere in $V_n$, so $\gamma(\mathcal{G}\cap V_n)=n$ by \cite[Proposition 44.10]{Zeidler1985}. Since $\mathcal{G}$ and $V_n$ are symmetric, their intersection $\mathcal{G}\cap V_n$ is 
symmetric as well. We have shown that $\mathcal{G}\cap V_n$ belongs to $\mathcal{G}_n$, which implies that $\mathcal{G}_n\neq \varnothing$ for all $n\in\mathbb{N}$. 
 
 Now, let $\mathcal{A}\in \mathcal{G}_n$. Since $F$ is continuous,  $F(\mathcal{A})$ is compact. In addition, $F$ is positive on $\mathcal{G}$. Therefore, $\inf_{u \in \mathcal{A}} F(u)>0$, which consequently yields $\beta_{n}>0$ for all $n\in \mathbb{N}$. Then the problem \eqref{operator eigenvalue problem} admits  eigenpairs $(\mu_n,\,u_n)$ with $\mu_{n}\neq 0$  and $F\left(u_n\right)=\beta_n$ by Theorem \ref{H1-H4 theorem}. 

Since  $\chi=\infty$ and $F$ satisfies \eqref{Zeidler 44.A (3)}, the eigenvalues $\mu_{n}$ are distinct by Theorem \ref{H1-H4 theorem}. We see that \eqref{mu n} follows from
\begin{equation*}
\mu_n=\mu_n G\left(u_n\right)=\frac{\mu_n}{p}\left\langle G^{\prime} (u_n), u_n\right\rangle=\frac{1}{p}\left\langle F^{\prime}(u_n), u_n\right\rangle=F\left(u_n\right)=\beta_n.
\end{equation*}
\end{proof}

\begin{corollary}\label{L-S eigenvalues}
$\lambda_n:=\frac{1}{\mu_n}$
is an eigenvalue of \eqref{eigenvalue problem} for each $n\in \mathbb{N}$. Moreover, $\lambda_n \to \infty$ as $n\to\infty$.
\end{corollary}

We can find these eigenvalues on  $\mathcal{F}$ defined by 
\begin{equation*}
\mathcal{F}:=\{u \in \mathcal{W}_{X,0}^{1, p}(\Omega):\enspace F(u)=1\}. 
\end{equation*}
Note that $\mathcal{F}$ is a  symmetric closed $C^1$-Finsler manifold. Indeed, since $F$ is even,  $\mathcal{F}$ is symmetric. Continuity of $F$ implies that $\mathcal{F}=F^{-1}(\{1\})$ is closed. Moreover, we have $\langle F^{\prime} (u), u\rangle=pF(u)$ for all $u\in\mathcal{W}_{X,0}^{1, p}(\Omega)$, so zero is the  only one critical value of $F$.  By the implicit function theorem, $\mathcal{F}$ is a $C^1$-Finsler manifold.
 
Let 
$$E(u):=\frac{G(u)}{F(u)}\quad \text{for } u\in \mathcal{W}_{X,0}^{1, p}(\Omega)\setminus\{0\}.$$
Given $u\in\mathcal{F}$, let 
$$
T_u \mathcal{F}:=\left\{v \in \mathcal{W}_{X, 0}^{1, p}(\Omega):\left\langle F^{\prime}(u), v\right\rangle=0\}\right.
$$
be the tangent space of $\mathcal{F}$ at $u$. The tangential mapping $\left(\left.E\right|_{\mathcal{F}}\right)^{\prime}(u)$ with respect to $\mathcal{F}$ at each $u\in M$ is given by
$$
E|_{\mathcal{F}}^{\prime}(u)[v]=\left\langle E^{\prime}(u), v\right\rangle \quad\text { for all } v \in T_u \mathcal{F}.
$$

\begin{definition}\label{Palais-Smale at c}
We say that the restriction $E|_{\mathcal{F}}$ satisfies the Palais--Smale condition at level $c \in \mathbb{R}$ on $\mathcal{F}$, in short $\mathrm{(PS)}_{c, \mathcal{F}}$, if every sequence $\{u_n\} \subset \mathcal{F}$ such that
\begin{enumerate}[label=(\roman*)]
\item $E(u_n) \to c$;
\item $\| (E|_{\mathcal{F}})'(u_n)\|_{(T_{u_n}\mathcal{F})^*} \to 0$,
\end{enumerate}
admits a convergent subsequence in $\mathcal{F}$.
\end{definition}
Given $u\in \mathcal{F}$, we define the linear continuous projection  $\mathcal{P}_u: \mathcal{W}_{X, 0}^{1, p}(\Omega) \rightarrow T_u\mathcal{F}$ explictly by 
$$
\mathcal{P}_u v:=v-\frac{\left\langle F^{\prime}(u), v\right\rangle}{\left\|F^{\prime}(u)\right\|_{\mathcal{W}_{X,0}^{1, p}(\Omega)^*}^2} J^{-1} F^{\prime}(u),
$$
where $J:\mathcal{W}_{X, 0}^{1, p}(\Omega)\to \mathcal{W}_{X, 0}^{1, p}(\Omega)^*$ is the duality mapping, see \cite[Proposition 47.19]{Zeidler1985}.
Then
$$
\left\|\mathcal{P}_u v\right\| \leq\|v\|+\frac{\left\|F^{\prime}(u)\right\|_{\mathcal{W}_{X,0}^{1, p}(\Omega)^*}\|v\|}{\left\|F^{\prime}(u)\right\|_{\mathcal{W}_{X,0}^{1, p}(\Omega)^*}^2}\left\|F^{\prime}(u)\right\|_{\mathcal{W}_{X,0}^{1, p}(\Omega)^*}=2\|v\|,
$$
where we have used $\left\|J^{-1} F^{\prime}(u)\right\|=\left\|F^{\prime}(u)\right\|_{\mathcal{W}_{X,0}^{1, p}(\Omega)^*}$.
Then
$$
\begin{aligned}
\langle G^{\prime}\left(u\right)-E\left(u\right) F^{\prime}\left(u\right), \mathcal{P}_u v\rangle&=\langle E|_{\mathcal{F}}^{\prime}(u), \mathcal{P}_uv\rangle \\ &\leq\|E|_{\mathcal{F}}^{\prime}(u)\|_{\left(T_{u} \mathcal{F}\right)^*}\|\mathcal{P}_u v\| 
\leq 2 \|E|_{\mathcal{F}}^{\prime}(u)\|_{\left(T_{u} \mathcal{F}\right)^*}\|v\|
\end{aligned}
$$
for all $v \in \mathcal{W}_{X, 0}^{1, p}(\Omega)$. 

\begin{lemma}\label{Palais Smale lemma}
$E|_{\mathcal{F}}$ satisfies $\mathrm{(PS)_{c, \mathcal{F}}}$ for all   $c\in\mathbb{R}$.
\end{lemma}
\begin{proof}
Let $\left\{u_n\right\} \subset \mathcal{F}$ be such that
\begin{enumerate}[label=(\roman*)]
\item $E(u_{n})\to c$;
\item $\left\|E|_{\mathcal{F}}^{\prime}(u_n)\right\|_{\left(T_{u_n} \mathcal{F}\right)^*} \rightarrow 0$
\end{enumerate}    
for some $c\in\mathbb{R}.$ Then  
\begin{equation}\label{some inequality}
\sup _{\|v\| \leq 1}\left\langle G^{\prime}(u_{n})-E(u_{n})F^{\prime}(u_{n}), P_{u_n} v\right\rangle \leq 2\sup _{\|v\| \leq 1} \|E|_{\mathcal{F}}^{\prime}(u)\|_{\left(T_{u} \mathcal{F}\right)^*}\|v\| \rightarrow 0 .
\end{equation}
Hence,
$$
\begin{aligned}
& G^{\prime}\left(u_n\right)-E\left(u_n\right) F^{\prime}\left(u_n\right)-\frac{\left\langle G^{\prime}\left(u_n\right)-E\left(u_n\right) F^{\prime}\left(u_n\right), J^{-1} F^{\prime}\left(u_n\right)\right\rangle}{\left\|F^{\prime}\left(u_n\right)\right\|^2} F^{\prime}\left(u_n\right) \\
& =G^{\prime}\left(u_n\right)-\frac{\left\langle G^{\prime}\left(u_n\right), J^{-1} F^{\prime}\left(u_n\right)\right\rangle}{\left\|F^{\prime}\left(u_n\right)\right\|^2} F^{\prime}\left(u_n\right)\to 0.
\end{aligned}
$$
since $\left\langle F^{\prime}\left(u_n\right), J^{-1} F^{\prime}\left(u_n\right)\right\rangle=\left\|F^{\prime}\left(u_n\right)\right\|^2$.

Our aim is to show that $\{u_n\}$ has a convergent subsequence. It is apparent to see that $\{u_n\}$ is bounded in $\mathcal{W}_{X,0}^{1, p}(\Omega)$, say $
\left\|u_n\right\| \leq M$.
 By Theorem \ref{compact Sobolev embedding},  we have a subsequence, still denoted by   $\{u_n\}$, such that  $u_n\rightharpoonup u$ in $\mathcal{W}_{X,0}^{1, p}(\Omega)$ and  $u_n\to u$ in $L^{p}(\Omega)$. Moreover, from the strong continuity of $F^{\prime}$, it follows that  $F^{\prime}(u_n)\to F^{\prime}(u)$ in $\mathcal{W}_{X,0}^{1, p}(\Omega)^{*}$. From
$$
\left\|F^{\prime}\left(u_n\right)\right\| \geq \frac{\left|\left\langle F^{\prime}\left(u_n\right), u_n\right\rangle\right|}{\left\|u_n\right\|}=\frac{p}{\left\|u_n\right\|}\geq \frac{p}{M}>0. 
$$
we see that $F^{\prime}\left(u_n\right)$ is bounded away from zero. Since $\left\|G^{\prime}(u)\right\|=p\|u\|^{p-1}$,  it follows that $G^{\prime}\left(u_n\right)$ is bounded. Therefore, there exists a subsequence, still denoted by   $G^{\prime}\left(u_{n}\right)$, such that 
$$
G^{\prime}\left(u_{n}\right) \rightharpoonup v \quad \text { in } \mathcal{W}_{X, 0}^{1, p}(\Omega)^* 
$$
for some $v\in\mathcal{W}_{X, 0}^{1, p}(\Omega)^*$. By \textbf{(H3)} and Proposition \ref{proposition H3}, it follows that  $u_n\to u$ in $\mathcal{W}_{X, 0}^{1, p}(\Omega)$.

\end{proof}
We apply the following theorem to obtain another characterization of $\{\lambda_n\}$.
\begin{theorem}\label{Szulkin theorem}\cite[Corollary 4.1]{Szulkin1988}
Let $\mathcal{M}$  be a closed symmetric $C^1$-submanifold of a real Banach space $\mathbb{B}$ such that $0 \notin \mathcal{M}$. Let  $f \in C^1(\mathcal{M}, \mathbb{R})$ be an even and bounded below functional. Given $n\in\mathbb{N}$, we define
$$
c_n:=\inf _{\mathcal{A} \in \Gamma_n} \sup _{u \in \mathcal{A}} f(u),
$$
where 
$$\Gamma_n:=\{\mathcal{A} \subset \mathcal{M}: \mathcal{A}\text{ is compact and symmetric, } \gamma(\mathcal{A}) \geq n\}.$$ 
If $\Gamma_n \neq \varnothing$ for some $n\in\mathbb{N}$ and  $f$ satisfies $(PS)_{c,\mathcal{M}}$ for each $c=c_k$ with $k=1, \ldots, n$, then there are at least $n$ distinct pairs of critical points of $f$.
\end{theorem}
Given $n\in\mathbb{N}$, we define
\begin{equation*}
\mathcal{F}_n:=\left\{\mathcal{A} \subset \mathcal{F}: \mathcal{A} \right. \text{ is compact and symmetric, }\left.\gamma(\mathcal{A}) \geq n\right\}
\end{equation*}
and 
\begin{equation*}
\nu_n:=\inf _{\mathcal{A} \in \mathcal{F}_n} \sup _{u \in \mathcal{A}} E(u).
\end{equation*}

\begin{theorem}
The number $\nu_n$ is an eigenvalue of \eqref{eigenvalue problem}. Namely,
$$\nu_n=\lambda_n\quad \text{for all } n\in\mathbb{N}.$$

\end{theorem}

\begin{proof}
First we show that $\mathcal{F}_n\neq \varnothing$ for every $n\in \mathbb{N}$. Let $W_n$ be $n$-dimensional subspace of $L^{p}(\Omega)$. Then $W_n$ is closed in $L^{p}(\Omega)$. Also, let $S:=\{u\in L^{p}(\Omega): \|u\|_{p}=1\}$. Then $W_n\cap S$ is closed and bounded, that is, $W_n\cap S$ is compact. Moreover, $W_n\cap S$ is the unit sphere in $W_{n}$, so $\gamma(W_n\cap S)=n$ by \cite[Proposition 44.10]{Zeidler1985}. Let $P:L^{p}(\Omega)\to L^{p}(\Omega)$ defined by
$$P(u):= \begin{cases}u & \text { if } u \in \mathcal{W}_{X,0}^{1, p}(\Omega), \\
0 & \text { if } u \notin \mathcal{W}_{X,0}^{1, p}(\Omega),\end{cases}$$
which is linear, continuous and odd. Then $P(W_n\cap S)=W_n\cap\mathcal{F}$ is compact, and $\gamma(W_n\cap\mathcal{F})\geq n$ by  \cite[Lemma 3.5(1)]{Rab1973}. Since $W_n$ is a subspace and $\mathcal{F}$ is symmetric, it follows that $W_n\cap\mathcal{F}$ is also symmetric. We have shown that $W_n\cap\mathcal{F}\in \mathcal{F}_n$ for all $n\in\mathbb{N}$, thus $\mathcal{F}_n\neq \varnothing$ for all $n\in\mathbb{N}$.

Since $E|_{\mathcal{F}}\in C^{1}(\mathcal{F},\mathbb{R})$ is bounded below and $E|_{\mathcal{F}}$ satisfies $(P S)_{c, \mathcal{F}}$ for every $c=\nu_n$ with $n\in\mathbb{N}$, it follows from Theorem \ref{Szulkin theorem} that there is a sequence of critical points of $E|_{\mathcal{F}}$. Moreover,
$$\inf _{\mathcal{A} \in \mathcal{F}_n} \sup _{u \in \mathcal{A}} E(u)=\frac{1}{\sup _{\mathcal{A} \in \mathcal{G}_n} \inf _{u \in \mathcal{A}} F(u)},$$
thus $\nu_n=\lambda_n$ for all $n\in\mathbb{N}$.
\end{proof}

\section{Special collection of symmetric, compact subsets on $\mathcal{F}$}

Let 
\begin{equation}\label{F_n special}
 \mathbb{F}_n:=\{\mathcal{A} \subset \mathcal{F}:\enspace\text{there exists a continuous odd surjection }  \left.h: \mathbb{S}^{n-1} \rightarrow \mathcal{A}\right\}, 
\end{equation}
where $\mathbb{S}^{n-1}$ is the unit sphere in $\mathbb{R}^{n}$. It is known that $\gamma(\mathbb{S}^{n-1})=n$, and subsequently by \cite[Lemma 3.5(1)]{Rab1973}, we have
\begin{equation*}
\gamma(h(\mathbb{S}^{n-1}))\geq n.    
\end{equation*}
Since $\mathbb{S}^{n-1}$ is symmetric and $h$ is an odd surjection, $\mathcal{A}$ is  symmetric. Moreover, $h(\mathbb{S}^{n-1})=\mathcal{A}$ is compact.

We show that $\mathbb{F}_n$ is nonempty for all $n\in\mathbb{N}$.  Since $\gamma(\mathbb{S}^{n-1})=n$,  there exists a continuous odd map $\widetilde{h}:\mathbb{S}^{n-1}\to\mathbb{R}^{n}\setminus\{0\}$. Moreover, $\dim \mathbb{R}^n=\dim W_n=n$ implies that there is an isomorphism $g:\mathbb{R}^{n}\to W_n$. It is immediate that $g$ is linear and bounded, therefore, $g:\mathbb{R}^{n}\setminus\{0\}\to W_n\setminus\{0\}$ is odd and continuous. Let $f:W_n\setminus\{0\}\to W_n\cap S$ defined by $f(u):=\|u\|_{p}^{-1}u$ for all $W_n\setminus\{0\}$. We see that $f$ is odd and continuous. Then
$$h:=P\circ f\circ g\circ \widetilde{h}:\mathbb{S}^{n-1}\to W_n\cap\mathcal{F}$$ is odd and continuous. It is surjective, provided  $\mathcal{A}=h(\mathbb{S}^{n-1})$.

Now, we define
\begin{equation}\label{gamma sequence}
\gamma_{n}:=\inf _{\mathcal{A} \in \mathbb{F}_n} \sup _{u \in \mathcal{A}} E(u).
\end{equation}

\begin{definition}
We say that $E|_{\mathcal{F}}$ satisfies the  Palais-Smale condition on $\mathcal{F}$, in short $\mathrm{(PS)_{\mathcal{F}}}$, if  $\left\{u_n\right\}$ is a sequence in $\mathcal{F}$ such that
\begin{enumerate}[label=(\roman*)]
\item $\{E(u_{n})\}$ is bounded;
\item  $\left\|\left.E\right|_{\mathcal{F}} ^{\prime}\left(u_n\right)\right\|_{\left(T_{u_n} \mathcal{F}\right)^*} \rightarrow 0$, 
\end{enumerate}
then it admits a convergent subsequence.
\end{definition}
Note that, $E|_{\mathcal{F}}$ satisfies  $\mathrm{(PS)_{\mathcal{F}}}$ if and only if $\mathrm{(PS)_{c,\mathcal{F}}}$ holds for all $c \in \mathbb{R}$. Therefore, $E|_{\mathcal{F}}$ also satisfies  $\mathrm{(PS)_{\mathcal{F}}}$ by Lemma \ref{Palais Smale lemma}.

Let us recall a version of the deformation lemma presented in \cite{DR1999}.  The original idea goes back to Rabinowitz, see e.g. \cite{Rabinowitz1978}.
\begin{lemma}\label{deformation lemma} 
Let $K$ be a $C^1$-functional on a complete connected $C^1$-Finsler manifold $\mathbb{M}$ satisfying the Palais-Smale condition.
If $\alpha \in \mathbb{R}$ is a regular value of $K$ and  $\bar{\varepsilon}>0$, then there exist  $\varepsilon \in(0, \bar{\varepsilon})$ and a continuous deformation $\psi: \mathbb{M} \times[0,1] \rightarrow \mathbb{M}$ such that the following conditions hold:
\begin{enumerate}[label=(\roman*)]
\item $\psi(\cdot,t)$ is a homeomorphism for every $t \in[0,1]$;
    \item $\psi(u, t)=u$,  if $|K(u)-\alpha \mid \geq \bar{\varepsilon}$ or if $t=0$;
    \item $K(\psi(u, t))$ is nonincreasing in $t$ for every $u \in \mathcal{F}$;
    \item If $K(u) \leq \alpha+\varepsilon$, then $K(\psi(u, 1)) \leq\alpha-\varepsilon$;
    \item $\psi$ is odd with respect to $u$ for any $t \in[0,1]$.
\end{enumerate} 
\end{lemma}

\begin{theorem}\label{gamma_n is a critical value}
The number $\gamma_{n}$ is a critical value of $E|_{\mathcal{F}}$.
\end{theorem}
\begin{proof}
Assume for contradiction that  $\gamma_{n}$ is not a critical value of $E|_{\mathcal{F}}$, that is, a regular value of $E|_{\mathcal{F}}$. Setting $\bar{\varepsilon}=1$ and $\alpha=\gamma_n$, let $\varepsilon \in(0, \bar{\varepsilon})$. Then there exists a continuous deformation $\psi: \mathcal{F} \times[0,1] \rightarrow \mathcal{F}$ by Lemma \ref{deformation lemma}. By definition of $\gamma_n$, there exists a set $\mathcal{A} \in \mathbb{F}_n$ such that 
$$\sup_{u\in \mathcal{A}}E(u)\leq \gamma_n+\varepsilon.$$
By definition of $\mathbb{F}_n$,  there exists a continuous odd surjection $h: \mathbb{S}^{n-1} \rightarrow \mathcal{A}$, so $\psi(h(\cdot),1):\mathbb{S}^{n-1}\to\psi(\mathcal{A},1) $ is also a continuous odd surjection. Hence, $\psi(\mathcal{A},1)\in \mathbb{F}_n$, and by Lemma \ref{deformation lemma},
$$\sup_{u\in \psi(\mathcal{A},1)}E(u)\leq \gamma_n-\varepsilon.$$
This is a contradiction, so $\gamma_{n}$ is a critical value of $E$ on $\mathcal{F}$.
\end{proof}

\begin{theorem}
The number $\gamma_n$ given by \eqref{gamma sequence} is an eigenvalue of \eqref{eigenvalue problem} for each $n\in\mathbb{N}$.
\end{theorem}
\begin{proof}
By Theorem \ref{gamma_n is a critical value},  $\gamma_{n}$ is a critical value of $E$ on $\mathcal{F}$, that is, there exist $\mathcal{A}\in\mathbb{F}_n$ and  $u_{n}\in\mathcal{A}$ such that $E(u_n)=\gamma_n$ and $E|_{\mathcal{F}}^{\prime}(u_{n})=0$. Since  $E|_{\mathcal{F}}^{\prime}=G^{\prime}-EF^{\prime}$, we see that $$E|_{\mathcal{F}}^{\prime}(u_{n})=G^{\prime}(u_{n})-E(u_{n})F^{\prime}(u_{n})=G^{\prime}(u_{n})-\gamma_{n}F^{\prime}(u_{n})=0.$$
\end{proof}
Finally, every $\mathcal{F}_n$ contains  $\mathbb{F}_n$,  so $
\lambda_n \leq \gamma_n$ for all $n\in\mathbb{N}$.
By Corollary \ref{L-S eigenvalues}, $\gamma_n\to \infty$  as $n\to\infty$.

\section{Appendix}
Assume that $k\in \mathbb{N}$.
\begin{lemma}\label{Lemma 5.1}\cite{DJM2001}
If $1<p\leq 2$, then there exists $C=C(p)>0$ such that
\begin{equation}\label{GM1975 p<2}
\left|\left| \omega_{1}\right|^{p-2} \omega_{1}-|\omega_{2}|^{p-2} \omega_{2}\right|\leq C| \omega_{1}-\left.\omega_{2}\right|^{p-1}  
\end{equation}
holds for all $\omega_{1}, \omega_{2} \in \mathbb{R}^{k}$.  If $p\geq 2$, then there exists $C=C(p)>0$ such that
\begin{equation}\label{GM1975 p>2}
\left|\left| \omega_{1}\right|^{p-2} \omega_{1}-|\omega_{2}|^{p-2} \omega_{2}\right|\leq C| \omega_{1}-\omega_{2} \mid(|\omega_{1}|+|\omega_{2}|)^{p-2}    
\end{equation}
holds for all $\omega_{1}, \omega_{2} \in \mathbb{R}^{k}$. 
\end{lemma}

\begin{lemma}\cite[Appendix]{Lin1990} \label{Lindqvist lemma}
If $1<p\leq 2$, then there exists $C=C(p)>0$ such that
\begin{equation}\label{vector inequality p<2}
\left(\left|\omega_{1}\right|^{p-2} \omega_{1}-\left|\omega_{2}\right|^{p-2} \omega_{2}\right) \cdot\left(\omega_{1}-\omega_{2}\right) \geq C \frac{\left|\omega_{1}-\omega_{2}\right|^2}{\left(\left|\omega_{2}\right|+\left|\omega_{1}\right|\right)^{2-p}}
\end{equation}
holds for all $\omega_{1}, \omega_{2} \in \mathbb{R}^{k}$. If $p\geq 2$, then there exists $C=C(p)>0$ such that
\begin{equation}\label{vector inequality p>2}
\left(|\omega_{1}|^{p-2} \omega_{1}-|\omega_{2}|^{p-2} \omega_{2}\right)\cdot(\omega_{1}-\omega_{2}) \geq C|\omega_{1}-\omega_{2}|^p 
\end{equation} 
holds for all $\omega_{1}, \omega_{2} \in \mathbb{R}^{k}$.
\end{lemma}

\begin{lemma}\label{Adams lemma}

If $1<p\leq 2$, then
\begin{equation}\label{Adams 1<p<2}
\left|\frac{\omega_{1}+\omega_{2}}{2}\right|^{p^{\prime}}+\left|\frac{\omega_{1}-\omega_{2}}{2}\right|^{p^{\prime}} \leq\left[\frac{1}{2}\left(|\omega_{1}|^p+|\omega_{2}|^p\right)\right]^{\frac{1}{p-1}}
\end{equation}
holds for all $\omega_{1}, \omega_{2} \in \mathbb{R}^{k}$.
If $p\geq 2$, then
\begin{equation}\label{Adams p>2}
\left|\frac{\omega_{1}+\omega_{2}}{2}\right|^p+\left|\frac{\omega_{1}-\omega_{2}}{2}\right|^p \leq \frac{1}{2}\left(|\omega_{1}|^p+|\omega_{2}|^p\right)
\end{equation}
holds for all $\omega_{1}, \omega_{2} \in \mathbb{R}^{k}$.
\end{lemma}
When $\omega_{1}, \omega_{2}\in\mathbb{C}$, Lemma \ref{Adams lemma} can be found in \cite[Lemma 2.37]{Adams2003Sobolev}.

\section{Acknowledgments}
This research is funded by the Science Committee of the Ministry of Science and Higher Education of  Kazakhstan (Grant No. AP19674900).  The author would like to thank his PhD advisor Prof. Durvudkhan Suragan for helpful discussions, and the  anonymous referees  for their constructive comments that improved the manuscript.

\section{Data availability}
The manuscript has no associated data.

\addcontentsline{toc}{chapter}{Bibliography}
%uncomment next line to change bibliography name to references
%\renewcommand{\bibname}{References}
\bibliography{refs}      %use a bibtex bibliography file refs.bib

\begin{thebibliography}{10}

\bibitem{Adams2003Sobolev}
R.~A. Adams and J.~J.~F. Fournier.
\newblock {\em Sobolev spaces}, volume 140 of {\em Pure and Applied Mathematics
  (Amsterdam)}.
\newblock Elsevier/Academic Press, Amsterdam, second edition, 2003.

\bibitem{Bonfiglioli2007}
A.~Bonfiglioli, E.~Lanconelli, and F.~Uguzzoni.
\newblock {\em Stratified {L}ie groups and potential theory for their
  sub-{L}aplacians}.
\newblock Springer Monographs in Mathematics. Springer, Berlin, 2007.

\bibitem{bramanti2023hormander}
M.~Bramanti and L.~Brandolini.
\newblock {\em H\"{o}rmander operators}.
\newblock World Scientific Publishing Co. Pte. Ltd., Hackensack, NJ, 2023.

\bibitem{capogna1993embedding}
L.~Capogna, D.~Danielli, and N.~Garofalo.
\newblock An embedding theorem and the {H}arnack inequality for nonlinear
  subelliptic equations.
\newblock {\em Communications in Partial Differential Equations},
  18(9-10):1765--1794, 1993.

\bibitem{capogna2024asymptotic}
L.~Capogna, G.~Giovannardi, A.~Pinamonti, and S.~Verzellesi.
\newblock The asymptotic $p$-{P}oisson equation as $p\to\infty$ in
  {C}arnot-{C}arath{\'e}odory spaces.
\newblock {\em Mathematische Annalen}, pages 1--41, 2024.

\bibitem{chow1940systeme}
W.-L. Chow.
\newblock {\"U}ber systeme von liearren partiellen differentialgleichungen
  erster ordnung.
\newblock {\em Mathematische Annalen}, 117(1):98--105, 1940.

\bibitem{Danielli1991}
D.~Danielli.
\newblock A compact embedding theorem for a class of degenerate {S}obolev
  spaces.
\newblock {\em Rend. Sem. Mat. Univ. Politec. Torino}, 49(3):399--420, 1991.

\bibitem{DJM2001}
G.~Dinca, P.~Jebelean, and J.~Mawhin.
\newblock Variational and topological methods for {D}irichlet problems with
  {$p$}-{L}aplacian.
\newblock {\em Port. Math. (N.S.)}, 58(3):339--378, 2001.

\bibitem{DR1999}
P.~Dr\'{a}bek and S.~B. Robinson.
\newblock Resonance problems for the {$p$}-{L}aplacian.
\newblock {\em J. Funct. Anal.}, 169(1):189--200, 1999.

\bibitem{AA1987}
J.~P. Garc\'{\i}a~Azorero and I.~Peral~Alonso.
\newblock Existence and nonuniqueness for the {$p$}-{L}aplacian: nonlinear
  eigenvalues.
\newblock {\em Comm. Partial Differential Equations}, 12(12):1389--1430, 1987.

\bibitem{KS2023}
M.~Karazym and D.~Suragan.
\newblock Subelliptic $p$-{L}aplacian spectral problem for {H}\"ormander vector
  fields.
\newblock {\em Mathematische Nachrichten}, 298(4):1184--1200, 2025.

\bibitem{Le2006}
A.~L\^{e}.
\newblock Eigenvalue problems for the {$p$}-{L}aplacian.
\newblock {\em Nonlinear Anal.}, 64(5):1057--1099, 2006.

\bibitem{LeDret2018}
H.~Le~Dret.
\newblock {\em Nonlinear elliptic partial differential equations}.
\newblock Universitext. Springer, Cham, 2018.

\bibitem{Leoni2017Sobolev}
G.~Leoni.
\newblock {\em A first course in {S}obolev spaces}, volume 181 of {\em Graduate
  Studies in Mathematics}.
\newblock American Mathematical Society, Providence, RI, second edition, 2017.

\bibitem{Lin1990}
P.~Lindqvist.
\newblock On the equation {${\rm div}\,(|\nabla u|^{p-2}\nabla
  u)+\lambda|u|^{p-2}u=0$}.
\newblock {\em Proc. Amer. Math. Soc.}, 109(1):157--164, 1990.

\bibitem{Perera2003}
K.~Perera.
\newblock Nontrivial critical groups in {$p$}-{L}aplacian problems via the
  {Y}ang index.
\newblock {\em Topol. Methods Nonlinear Anal.}, 21(2):301--309, 2003.

\bibitem{Rab1973}
P.~H. Rabinowitz.
\newblock Some aspects of nonlinear eigenvalue problems.
\newblock {\em Rocky Mountain J. Math.}, 3:161--202, 1973.

\bibitem{Rabinowitz1978}
P.~H. Rabinowitz.
\newblock Some minimax theorems and applications to nonlinear partial
  differential equations.
\newblock In {\em Nonlinear analysis (collection of papers in honor of {E}rich
  {H}. {R}othe)}, pages 161--177. Academic Press, New York-London, 1978.

\bibitem{rashevsky1938any}
P.~K. Rashevsky.
\newblock Any two points of a totally nonholonomic space may be connected by an
  admissible line.
\newblock {\em Uchenye Zapiski MGPI im. K. Liebknechta, seriya
  fiziko-matematicheskaya}, 2:pp. 83--94 (in Russian), 1938.

\bibitem{Szulkin1988}
A.~Szulkin.
\newblock Ljusternik-{S}chnirelmann theory on {${\it C}^1$}-manifolds.
\newblock {\em Ann. Inst. H. Poincar\'{e} Anal. Non Lin\'{e}aire},
  5(2):119--139, 1988.

\bibitem{xu1990subelliptic}
C.~J. Xu.
\newblock Subelliptic variational problems.
\newblock {\em Bull. Soc. Math. France}, 118(2):147--169, 1990.

\bibitem{Zeidler1985}
E.~Zeidler.
\newblock {\em Nonlinear functional analysis and its applications. {III}}.
\newblock Springer-Verlag, New York, 1985.
\newblock Variational methods and optimization, Translated from the German by
  Leo F. Boron.

\end{thebibliography}

\bibliographystyle{abbrv}  %use the plain bibliography style unsrt apalike siam

\end{document}